\documentclass{amsart}

\usepackage{amssymb,amsthm}

\usepackage{enumerate}

\usepackage{graphicx}

\newcommand{\C}{\ensuremath{\mathbb{C}}}%
\newcommand{\Z}{\mathbb{Z}}

\newcommand{\Gr}{\mathcal{G}}
\newcommand{\Hr}{\mathcal{H}}
\newcommand{\alb}{\mathsf{X}}
\newcommand{\xo}{\alb^\omega}
\newcommand{\xs}{\alb^*}
\newcommand{\tail}{\mathcal{T}}
\newcommand{\paths}{\Omega}
\newcommand{\symm}{\mathop{\mathrm{Symm}}}
\newcommand{\supp}{\mathop{\mathrm{supp}}}
\newcommand{\aut}{\mathop{\mathrm{Aut}}}
\newcommand{\X}{\mathcal{X}}
\newcommand{\be}{\mathsf{o}}
\newcommand{\en}{\mathsf{t}}
\newcommand{\arr}{\longrightarrow}

\usepackage{color}

\newtheorem{theorem}{Theorem}[section]
\newtheorem{proposition}[theorem]{Proposition}

\newtheorem{lemma}[theorem]{Lemma}

\theoremstyle{definition}
\newtheorem{defi}[theorem]{Definition}
\newtheorem{examp}[theorem]{Example}
\newtheorem*{remark}{Remark}

\title{Extensions of amenable groups by recurrent groupoids}
\author{Kate Juschenko, Volodymyr Nekrashevych, Mikael de la Salle}

\begin{document}
\thanks{V. Nekrashevych was supported by NSF grant DMS1006280. M.~de la Salle was supported by ANR grants OSQPI and NEUMANN}
\begin{abstract}
We show that the amenability of a group acting by homeomorphisms can be deduced from a certain local property of the action and recurrency of the orbital Schreier graphs. This applies to a wide class of groups, the amenability of which was an open problem, as well as unifies many known examples to one general proof. In particular, this includes Grigorchuk's group, Basilica group, the full topological group of Cantor minimal system, groups acting on rooted trees by bounded automorphisms, groups generated by finite automata of linear activity growth, groups that naturally appear in holomorphic dynamics.

\end{abstract}
\maketitle

\section{Introduction}

M.~Day introduced the class EG of \emph{elementary amenable groups}
in~\cite{day:semigroups} as the class of all groups that can be
constructed from finite and abelian using operations of passing to a
subgroup, quotient, group extensions, and direct limits (the fact that
the class of amenable groups is closed under these operations was
already proved by J.~von~Neumann). He notes that
at that time no examples of amenable groups that do not belong to the
class EG were known.

The first example of an amenable group not belonging to EG was the
Grigorchuk group of intermediate
growth~\cite{grigorchuk:milnor_en}. An example of a group which can
not be constructed from groups of sub-exponential growth (which, in
some sense, can be also considered as an ``easy case'' of amenability)
is the basilica group introduced
in~\cite{zukgrigorchuk:3st}. Its amenability was
proved in~\cite{barthvirag} using asymptotic properties of random
walks on groups. These methods were then generalized
in~\cite{bkn:amenability} and~\cite{amirangelvirag:linear} for a big class
of groups acting on rooted trees.
The first examples of finitely generated infinite simple amenable
groups (which also can not belong to EG) were constructed
in~\cite{juschenkomonod}.

A common feature of all known examples of non-elementary amenable
groups is that they are defined as groups of homeomorphisms of the
Cantor set (or constructed from such groups).

The aim of this paper is to show a general method of proving
amenability for a wide class of groups acting on topological
spaces. This class contains many new examples of groups, whose
amenability was an open question. It also contains all of the
mentioned above non-elementary amenable groups as simple examples.

We show that amenability of a group of homeomorphisms can be deduced
from a combination of local topological information about the
homeomorphisms and global information about the orbital Schreier
graphs. Namely, we prove the following amenability condition (see
Theorem~\ref{th:amenhomeo}). If $G$ is a group acting by
homeomorphisms on a topological space $\X$, then
by $[[G]]$ we denote the full topological group of the action, i.e.,
the group of all homeomorphisms $h$ of $\X$ such that
for every $x\in\X$ there exists a neighborhood of $x$ such that
restriction of $h$ to that neighborhood is equal to restriction of an
element of $G$. For $x\in\X$ the \emph{group of germs} of $G$ at $x$ is the
quotient of the stabilizer of $x$ by the subgroup of elements acting
trivially on a neighborhood of $x$.

\begin{theorem}
\label{th:amenhomeo1}
Let $G$ and $H$ be groups of homeomorphisms of
a compact topological space $\X$, and $G$ is finitely generated.
Suppose that the following conditions hold.
\begin{enumerate}
\item The full group $[[H]]$ is amenable.
\item For every element $g\in G$, the set of points $x \in X$ such that $g$ does not coincide with an element of $H$ on any neighborhood of $x$ is finite.
\item For every point $x\in\X$ the Schreier graph of the action of $G$
  on the orbit of $x$ is recurrent.
\item For every $x\in\X$ the group of germs of $G$ at $x$ is amenable.
\end{enumerate}
Then the group $G$ is amenable. Moreover, the group $[[G]]$ is amenable.
\end{theorem}

A key tool of the proof of the theorem is the following fact
(see Theorem~\ref{prop=equiv_sobolev}).

\begin{theorem}
Let $G$ be a finitely generated group acting on a set
$X$. If the graph of the action of $G$ on $X$ is recurrent, then there
exists an $(\Z/2\Z)\wr_X G$-invariant mean on $\bigoplus_X\Z/2\Z$.
\end{theorem}

The last two sections of the paper are devoted to showing different
examples of applications of Theorem~\ref{th:amenhomeo1}. At first we
consider the case when the group $[[H]]$ is locally finite. Examples
of such groups $[[H]]$ are given by \emph{block-diagonal direct limits} of
symmetric groups defined by a \emph{Bratteli diagram}. The
corresponding groups $G$ satisfying the conditions of
Theorem~\ref{th:amenhomeo1} are generated by \emph{homeomorphisms of
  bounded type}. 
  
 The class of groups generated by homeomorphisms of bounded type
includes some known examples of non-elementary amenable groups (groups
of bounded automata~\cite{bkn:amenability} and topological full group
of minimal homeomorphisms of the Cantor set~\cite{juschenkomonod}), as
well as groups whose amenability was an open question (for instance groups of
arbitrary bounded automorphisms of rooted trees).

The last section describes some examples of application of
Theorem~\ref{th:amenhomeo1} in the case when $H$ is not locally
finite. For example, one can use Theorem~\ref{th:amenhomeo1} twice:
prove amenability of $[[H]]$ using it, and then construct new amenable
groups $G$ using $H$. This way one gets a simple proof of
the main result of~\cite{amirangelvirag:linear}: that groups generated
by finite automata of linear activity growth are amenable. Two other
examples from Section~\ref{s:unbounded} are new, and are groups
naturally appearing in holomorphic dynamics. One is a holonomy group
of the stable foliation of the Julia set of a H\'ennon map, the other is
the iterated monodromy group of a mating of two quadratic
polynomials.

\subsection*{Acknowledgments} We thank Omer Angel and Laurent Bartholdi for useful comments on a previous version of this paper.

\section{Amenable and recurrent $G$-sets}

\subsection{Amenable actions}

\begin{defi}
Let $G$ be a discrete group. An action of $G$ on a set $X$ is said to be
\emph{amenable} if there exists an \emph{invariant mean} on $X$. Here
an invariant mean is a map $\mu$ from the set of all subsets of $X$ to
$[0, 1]$ such that $\mu$ is finitely additive, $\mu(X)=1$, and
$\mu(g(A))=\mu(A)$ for all $A\subset X$ and $g\in G$.
\end{defi}

A group $G$ is amenable if and only if its action on itself by left
multiplication is amenable.
Note that the above definition of amenability is different from another definition of
amenability of an action, due to Zimmer, see~\cite{zimmer:ergodictheory}.

The following criteria of amenability are proved in the same way as
the corresponding classical criteria of amenability of groups.

\begin{theorem}
\label{th:amenabcriteria}
Let $G$ be a discrete group acting on a set $X$. Then the following conditions
are equivalent.
\begin{enumerate}
\item The action of $G$ on $X$ is amenable.
\item \textbf{Reiter's condition.} For every finite subset $S\subset
  G$ and for every $\epsilon>0$ there exists a non-negative function
  $\phi\in\ell^1(X)$ such that $\|\phi\|_1=1$ and $\|g\cdot\phi-\phi\|_1<\epsilon$ for
  all $g\in S$.
\item \textbf{F\o lner's condition.} For every finite subset $S\subset
  G$ and for every $\epsilon>0$ there exists a finite subset $F\subset
  X$ such that $\sum_{g\in S}|gF\Delta F|<\epsilon |F|$.
\end{enumerate}
\end{theorem}

Note that if $G$ is finitely generated, it is enough to check
conditions (2) and (3) for one generating set $S$. If the conditions
of (3) hold, then we say that $F$ is $(S, \epsilon)$-F\o lner set, and
that $F$ is $\epsilon$-invariant with respect to the elements of $S$.

The following proposition is well known and follows from the fact that
a group is amenable if and only if it admits an action which is
amenable and Zimmer amenable, see \cite{rose}.

\begin{proposition}
\label{pr:amenableactionandstabilizer}
Let $G$ be a group acting on a set $X$. If the action is amenable and
for every $x\in X$ the stabilizer $G_x$ of $x$ in $G$ is an amenable
group, then the group $G$ is amenable.
\end{proposition}

\subsection{Recurrent actions}

Let $G$ be a finitely generated group acting transitively on a set $X$.
Choose a measure $\mu$ on $G$ such that support of $\mu$ is a finite
generating set of $G$ and $\mu(g)=\mu(g^{-1})$ for all $g\in
G$. Consider then the Markov chain on $X$ with transition probability
from $x$ to $y$ equal to $p(x, y)=\sum_{g\in G, g(x)=y}\mu(g)$.

The Markov chain is called \emph{recurrent} if the probability of ever
returning to $x_0$ after starting at $x_0$ is equal to 1 for some
(and hence for every) $x_0\in X$.

It is well known (see~\cite[Theorems~3.1,
3.2]{woess:rw}) that recurrence of the described Markov chain does not
depend on the choice of the measure $\mu$, if the measure is
symmetric, and has finite support generating the group. We say that the \emph{action} of
$G$ on $X$ is \emph{recurrent} if the corresponding Markov chain is
recurrent. Every recurrent action is amenable, see \cite[Theorem~10.6]{woess:rw}.

For a finite symmetric generating set $S$ the \emph{Schreier graph} $\Gamma(X, G, S)$ is
the graph with the set of vertices identified with $X$, the set of
edges is $S\times X$, where an edge $(s, x)$ connects $x$ to $s(x)$.

A transitive action of $G$ on $X$ is recurrent if and only if the
simple random walk on the Schreier graph $\Gamma(X, G, S)$ is recurrent.

The following fact is also a corollary
of~\cite[Theorem~3.2]{woess:rw}.

\begin{lemma}
\label{lem:recurrentamenablesubgroup}
If the action of a finitely generated group $G$ on a set $X$ is
recurrent, and $H<G$ is a subgroup, then the action of $H$ on every
its orbit on $X$ is also recurrent.
\end{lemma}

The following theorem is a straightforward corollary of the Nash-Williams criterion
(see~\cite{nashwilliams} or~\cite[Corollary~2.20]{woess:rw}).

\begin{theorem}
\label{th:nashwilliams}
Let $\Gamma$ be a connected graph of uniformly bounded degree with set
of vertices $V$. Suppose that there exists an increasing sequence of finite subsets $F_n\subset V$
such that $\bigcup_{n\ge 1}F_n=V$, $\partial F_n$ are disjoint subsets and
\[\sum_{n\ge 1}\frac{1}{|\partial F_n|}=\infty,\]
where $\partial F_n$ is the set of vertices of $F_n$ adjacent to the
vertices of $V\setminus F_n$. Then the simple random walk on $\Gamma$
is recurrent.
\end{theorem}

We will also use a characterization of transience of a random walk on
a locally finite connected graph $(V, E)$ in terms of electrical
network. The \emph{capacity} of a point $x_0 \in V$ is the quantity
defined by
\[\mathop{cap}(x_0) = \inf\left\{\left(\sum_{(x,x') \in E} |a(x) - a(x')|^2\right)^{1/2}\right\}\]
where the infimum is taken over all finitely supported functions $a:V
\to \C$ with $a(x_0)=1$. We will use the following

\begin{theorem}[\cite{woess:rw}, Theorem 2.12]\label{thm=woess} The
random walk on a locally finite connected graph $(V,E)$ is transient
if and only if $\mathop{cap}(x_0)>0$ for some (and hence for all)
$x_0 \in V$.
\end{theorem}

\subsection{A mean on $\bigoplus_X\Z_2$ invariant with respect to
  $\Z_2\wr_X G$}

Let $G$ be a discrete group acting transitively on a set $X$. Let
$\{0,1\}^X$ be the set of all subsets of $X$ considered as abelian
group with multiplication given by the symmetric difference of
sets. It is naturally isomorphic to the Cartesian product $\Z_2^X$,
where $\Z_2=\Z/2\Z$. The group $G$ acts naturally on the group
$\Z_2^X$ by automorphisms.

Denote by $\mathcal{P}_f(X)$ the subgroup of $\Z_2^X=\{0,1\}^X$ which
consists of all finite subsets of $X$, i.e., the subgroup
$\bigoplus\limits_X \Z_2$ of $\Z_2^X$. It is obviously invariant under
the action of $G$.

Consider the restricted wreath product $\Z_2\wr_X G\cong
G\ltimes\mathcal{P}_f(X)$. Its action on $\mathcal{P}_f(X)$ is given
by the formula $$(g, E)(F)=g(E\Delta F)$$ for $E, F\in
\mathcal{P}_f(X)$ and $g\in G$.

The Pontryagin dual of $\mathcal{P}_f(X)$ is the compact group
$\Z_2^X$, with the duality given by the pairing
$\phi(E,\omega)=exp(i\pi \sum\limits_{j\in E}\omega_j)$, $E\in
\mathcal{P}_f(X)$, $\omega\in\Z_2^X$. Fix a point $p\in X$ and denote
by $L_2(\Z_2^X,\mu)$ the Hilbert space of functions on $\Z_2^X$ with
the Haar probability measure $\mu$. Denote by $A_p=\{(\omega_x)_{x\in
  X}\in \Z_2^X:\omega_{p}=0\}$ be the cylinder set which fixes
$\omega_p$ as zero.

The following lemma was proved in \cite[Lemma~3.1]{juschenkomonod}.

\begin{lemma}\label{duality} Let $G$ acts transitively on a set $X$
  and choose a point $p\in X$. The following are equivalent:
\begin{enumerate}[(i)]
\item There exist a net of unit vectors $\{f_n\}\in L_2(\{0,1\}^X,\mu)$ such that for every $g\in G$
$$\|g\cdot f_n-f_n\|_2\rightarrow 0  \text{ and } \|f_n\cdot \chi_{A_p}\|_2\rightarrow 1.$$ \label{functions}
\item The action of $\Z_2\wr_X G$ on $\mathcal{P}_f(X)$ is amenable.
\item The action of $G$ on $\mathcal{P}_f(X)$ admits an invariant mean
  giving full weight to the collection of sets containing $p$.
\end{enumerate}
\end{lemma}

We say that a function $f\in L_2(\Z_2^X, \mu)$ is \emph{p.i.r.\ } if
it is a product of random independent variables, i.e., there are
functions $f_x\in\mathbb{R}^{\Z_2}$ such that
$f(\omega)=\prod\limits_{x\in X} f_x(\omega_x)$. In other words, if we
consider $L_2(\Z_2^X, \mu)$ as the infinite tensor power
of the Hilbert space $L_2(\Z_2, m)$ with unit vector $1$, where
$m(\{0\})=m(\{1\})=1/2$, the condition p.i.r.\ means that $f$ is an
elementary tensor in $L_2(\Z_2^X, \mu)$.

\begin{theorem}
\label{prop=equiv_sobolev}
Let $G$ be a finitely generated group acting transitively
on a set $X$. There exists a sequence of p.i.r.\ functions $\{f_n\}$
in $L_2(\Z_2^X, \mu)$ that satisfy condition (i) in Lemma~\ref{duality} if and only if
the action of $G$ on $X$ is recurrent.
\end{theorem}

\begin{proof}
Denote by $(X,E)$ the Schreier graph of the action of $G$ on $X$ with respect to $S$.
Suppose that the random walk on $(X,E)$ is recurrent. By Theorem
\ref{thm=woess} there exists $a_n=(a_{x, n})_x$ a
sequence of finitely supported functions such that $a_{x_0, n}=1$ and
$\sum_{x \sim x'} |a_{x, n}-a_{x', n}|^2 \to 0$. Replacing all values
$a_{x, n}$ that are greater than 1 (smaller than $0$) by 1
(respectively $0$) does not increase the
differences $|a_{x, n}-a_{x', n}|$, hence we may assume that $0 \leq
a_{x, n}\leq 1$. For $0\leq t \leq 1$ consider the unit vector $\xi_t \in L_2(\{0,1\},m)$
\[(\xi_t(0),\xi_t(1)) = (\sqrt 2 \cos(t\pi/4),\sqrt 2\sin(t \pi/4
)).\]
Define $f_{x,n} =\xi_{1-a_{x,n}}$ and $f_n=\bigotimes_{x\in X}
f_{x,n}$. We have to show that $\langle g f_n,f_n\rangle \to 1$ for
all $g \in \Gamma$. It is sufficient to show this for $g \in S$. Then
\begin{multline*}
\langle g f_n,f_n\rangle=\prod_x \langle f_{x,n},f_{g
  x,n}\rangle  = \prod_x \cos \frac{\pi}{4}(a_{x,n} - a_{gx,n}) 
 \geq\\ \prod_x e^{-\frac{\pi^2}{16} (a_{x,n} - a_{gx,n})^2} 
\geq e^{-\frac{\pi^2}{16}\sum_{x \sim x'} |a_{x,n} - a_{x',n}|^2}
\end{multline*}
and the last value converges to $1$. We used that $\cos(x) \geq
e^{-x^2}$ if $|x|\leq \pi/4$.

Let us prove the other direction of the theorem.
Define the following pseudometric on the unit sphere of $L_2(\{0,1\},m)$
by $$d(\xi,\eta) = \inf\limits_{{\omega \in \mathbb C, |\omega|=1}} \|
\omega \xi - \eta\| = \sqrt{2-2|\langle \xi,\eta\rangle|}.$$

Assume that there is a sequence of  p.i.r.\ functions $\{f_n\}$ in
$L_2(\Z_2^X, \mu)$ that satisfy condition \eqref{functions} of
Lemma~\ref{duality}. Write $f_n(\omega) = \prod_{x \in X}
f_{n,x}(\omega_x)$. We can assume that the product is
finite. Replacing $f_{n,x}$ by $f_{n,x}/\|f_{n,x}\|$ we can assume
that $\|f_{n,x}\|_{L_2(\Z_2, m)}=1$. Define $a_{x,n} =
d(f_{x,n},1)$. It is straightforward that $(a_{x,n})_{x \in X}$ has
finite support and $$\lim\limits_n a_{x_0,n}= d(\sqrt 2
\delta_{0},1)>0.$$ Moreover for every $g \in G$
\[|\langle g f_n,f_n\rangle| = \prod_x |\langle f_{n,x},f_{n,g
  x}\rangle| = \prod_x (1-d(f_{n,x},f_{n,gx})^2/2) \leq e^{-\sum_x
  d(f_{n,x},f_{n,gx})^2/2},\]
which by assumption goes to $1$ for every $g \in S$. By definition of
the Schreier graph and the triangle inequality for $d$,
\begin{align*}
\sum_{(x, x')\in E} |a_{x,n} - a_{x',n}|^2 &= \sum_{g \in S} \sum_x |a_{x,n} - a_{g x,n}|^2 
\leq \sum_{g \in S} \sum_x d(f_x,f_{gx})^2 \to 0.
\end{align*}
This proves that $\mathop{cap}(x_0)=0$ in $(X,E)$, and hence by
Theorem \ref{thm=woess} that the random walk on $(X,E)$ is recurrent.
\end{proof}

\section{Amenability of groups of homeomorphisms}
\label{s:finext}

\subsection{Groupoids}
Let $G$ be a group acting faithfully by homeomorphisms on a
topological space $\X$. A \emph{germ} of the action is an equivalence
class of pairs $(g, x)\in G\times\X$, where two germs $(g_1, x_1)$ and
$(g_2, x_2)$ are equal if and only if $x_1=x_2$, and there exists a
neighborhood $U$ of $x_1$ such that $g_1|_U=g_2|_U$. The set of all germs of the action of $G$ on $\X$ is a
\emph{groupoid}. Denote by $\be(g, x)=x$ and $\en(g, x)=g(x)$ the
\emph{origin} and \emph{target} of the germ. A composition $(g_1,
x_1)(g_2, x_2)$ is defined if $g_2(x_2)=x_1$, and then it is equal to
$(g_1g_2, x_2)$. The inverse of a germ $(g, x)$ is the germ $(g,
x)^{-1}=(g^{-1}, g(x))$.

The groupoid of germs has a natural topology defined by the basis of
open sets of the form $\{(g, x)\;:\;x\in U\}$, where $g\in G$ and
$U\subset\X$ is open. For a given groupoid $\Gr$ of germs of an action on $\X$, and for
$x\in\X$, the \emph{isotropy group}, or \emph{group of germs} $\Gr_x$ is the group of all germs
$\gamma\in\Gr$ such that $\be(\gamma)=\en(\gamma)=x$. If $\Gr$ is the groupoid of germs of the action of $G$ on $\X$,
then the isotropy group $\Gr_x$ is the quotient of the
stabilizer $G_x$ of $x$ by the subgroup of elements of $G$ that act
trivially on a neighborhood of $x$.

The \emph{topological full group} of a groupoid of germs $\Gr$, denoted
$[[\Gr]]$ is the set of all homeomorphisms $F:\X\arr\X$ such that all
germs of $F$ belong to $\Gr$. The \emph{(orbital) Schreier graph} $\Gamma(x, G)$ is the Schreier
graph of the action of $G$ on the $G$-orbit of $x$.

\subsection{Amenability of groups}

\begin{theorem}
\label{th:amenhomeo}
Let $G$ be a finitely generated group of homeomorphisms of
a topological space $\X$, and $\Gr$ be its groupoid of germs. Let $\Hr$ be a
groupoid of germs of homeomorphisms of $\X$.
Suppose that the following conditions hold.
\begin{enumerate}
\item The group $[[\Hr]]$ is amenable.
\item For every generator $g$ of $G$ the set of points $x\in\X$ such that
$(g, x)\notin\Hr$ is finite. We say that $x\in\X$ is \emph{singular}
if there exists $g\in G$ such that $(g, x)\notin\Hr$.\label{C2}
\item For every singular point $x\in\X$
the orbital Schreier graph $\Gamma(x, G)$ is recurrent.
\item The isotropy groups $\Gr_x$ are amenable.
\end{enumerate}
Then the group $G$ is amenable.
\end{theorem}

\begin{proof}
After replacing $\Hr$ by $\Hr\cap\Gr$, we may assume that $\Hr\subset\Gr$.
Let $S$ be a finite symmetric generating set of $G$.
Let $\Sigma$ be the set of points $x\in\X$ such that there exists
$g\in S$ such that $(g, x)\notin\Hr$. Let $V$ be the union of the $G$-orbits of the elements of $\Sigma$.
By assumption \ref{C2} every $G$-orbit is a union of a finite number of $\Hr$-orbits. Since
the set $\Sigma$ is finite, the set $V$ is a union of a finite number
of $\Hr$-orbits.

\begin{lemma}
\label{lem:singular}
The set $V$ contains all singular points of $\X$.
\end{lemma}

\begin{proof}
Let $g=g_1g_2\cdots g_n$ be a representation of $g$ as a product of
generators $g_i\in S$. The germ $(g, x)$ is then equal to the product
of germs
\begin{equation*}
\label{eq:germs}
(g_1, g_2\cdots g_n(x))\cdot (g_2, g_3\cdots g_n(x))\cdots (g_n,
x)
\end{equation*}
If all these germs belong to $\Hr$, then $(g, x)$ belongs to
$\Hr$. Therefore, $(g, x)\notin\Hr$ only if $x\in\Sigma\cup
g_n^{-1}\Sigma\cup (g_{n-1}g_n)^{-1}\Sigma\cup\cdots (g_2\cdots g_n)^{-1}\Sigma$.
\end{proof}

Let $A\subset V$ be an $\Hr$-orbit transversal. For every $v\in V$
there exists a unique element $A$ that belongs to the same $\Hr$-orbit
as $v$. Let us denote it by $\alpha(v)$. Choose a germ
$\delta_v\in\Hr$ such that
$\be(\delta_v)=\alpha(v)$ and $\en(\delta_v)=v$. 
For $g\in G$ and $v\in V$ denote by $g_v$ the element of $\Gr$ defined by
\begin{equation}
\label{eq:gv}
g_v=\delta_{g(v)}^{-1}(g,v)\delta_v,
\end{equation}
and note that it satisfies the cocycle relation $(gg')_v = g_{g'(v)} g'_v$.

Let $\Gr|_A$ (resp.\ $\Hr|_A$) be the set of germs $\gamma\in\Gr$
(resp.\ $\gamma\in\Hr$) with the target and the origin in $A$. Note
that $\Hr|_A$ is the disjoint union of the isotropy groups $\Hr_a$ for $a
\in A$, and that $g_v \in \Gr|_A$ for all $g \in G$ and $v \in V$.
Consider the quotient $Z=\Gr|_A/\Hr|_A$ of $\Gr|_A$ defined by
the right action of $\Hr|_A$, i.e., two germs $\gamma_1,
\gamma_2\in \Gr|_A$ are equivalent if there
exists $\gamma\in\Hr$ such that $\gamma_2=\gamma_1\gamma$. Note that
then $\en(\gamma_1)=\en(\gamma_2)$ and $\be(\gamma_1)=\be(\gamma_2)$, hence the maps $\en:Z\arr A$ and $\be:Z \to A$ are well defined.

Let $\mathcal{P}$ be the set of functions
$\phi:V\arr Z$ such that $t(\phi(v))=\alpha(v)$ for all $v \in V$ and of \emph{finite
  support}, i.e., such that the values of $\phi(v)$ are trivial (i.e.,
belong to $\Hr_{\alpha(v)}$) for all but a finite number of values $v\in V$.

For $\phi \in \mathcal P$, $g \in G$ and $v \in V$, define $g(\phi)(v)
= g_{g^{-1}v} \cdot (\phi( g^{-1} v))$. By assumption \ref{C2} of
the theorem, $g(\phi)$ belongs to $\mathcal P$ and this defines an action of $G$ on
$\mathcal P$ by the cocycle relation. 

\begin{proposition}
There exists a $G$-invariant mean on $\mathcal{P}$.
\end{proposition}

\begin{proof}
If we decompose $V=\bigcup_i V_i$ as a finite union of $G$-orbits, we get
a decomposition of $\mathcal P$ as a direct product of $\mathcal P_i$
where $\mathcal P_i$ are the restrictions of elements of $\mathcal P$
to $V_i$, and $G$ acts diagonally. It is therefore enough to prove the
proposition for the case when $G$ acts transitively on $V$.  

For every pair of elements $a,b \in A$, choose an element $f_{a,b} \in \Gr$ such that $\be(f_{a,b})=a$ and $\en(f_{a,b})=b$. We also assume that $f_{a,a}$ is the identity of $\Gr_a$. For every $\gamma \in \Gr|_A$ consider the element $\widetilde \gamma \in \coprod_{a \in A} \Gr_a$ (disjoint union) defined by $\widetilde \gamma = f_{\en(\gamma),\be(\gamma)} \gamma \in \Gr_{\be(\gamma)}$. We also denote by $\widetilde \cdot$ the induced map $\Gr|_A /\Hr|_A \to \coprod_{a \in A} \Gr_a/\Hr_a$.

For every $\phi \in \mathcal P$, consider the map $\psi\colon V \to
\coprod_{a \in A} \Gr_a/\Hr_a$ defined by $\psi(v) =
\widetilde{\phi(v)}$. The map $\phi \mapsto \psi$ allows to identify
$\mathcal P$ with the set $\widetilde{\mathcal P}$ of functions $\psi$
from $V$ to $\coprod_{a \in A} \Gr_a/\Hr_a$ such that $\psi(v) =
\Hr_{\alpha(v)}$ for all but finitely many $v$'s. 

One easily checks that the action of $G$ on $\widetilde{\mathcal{P}}$ using this identification is given by
\begin{equation}
\label{eq:gactionphi}
(g \cdot \psi)(g v)= f_{\alpha(g v),\be(\psi(v))} g_v
f_{\alpha(v),\be(\psi(v))}^{-1}\psi(v),
\end{equation}
where $g_v$ is given by~\eqref{eq:gv}. If the germ $(g,v)$ belongs to $\Hr$, then $\alpha(v) = \alpha( g v)$ and $g_v \in \Hr_{\alpha(v)}$. If additionally $\psi(v)$ is trivial (i.e. is equal to $\Hr_{\alpha(v)}$), then so is $(g \cdot \psi)(g v)$ by our choice of $f_{a,a}=1$.

By Lemma~\ref{duality} and Theorem~\ref{prop=equiv_sobolev} there
exists a $G$-invariant mean on $\mathcal{P}_f(V)$ giving full weight
to the collection of sets containing a given point $p\in V$. Note that
since the mean is $G$-invariant, finitely additive, and the action of
$G$ on $V$ is transitive, the mean gives full weight to the collection
of sets containing any given finite subset of $V$. In particular, it
gives full weight to the collection of sets containing $\Sigma$.

It follows then from Theorem~\ref{th:amenabcriteria} that
for every $\epsilon>0$ there exists a finite
subset $\mathcal{F}$ of $\mathcal{P}_f(V)$ which is
$\epsilon$-invariant under the action of elements of $S$, such
that every element of $\mathcal{F}$ contains $\Sigma$. Moreover, since
$G$ preserves cardinalities of elements of $\mathcal{P}_f(V)$, we can
choose $\mathcal{F}$ consisting of sets of the same cardinality $N$.

Fix $x_0 \in A$. Let us assume at first that $\Gr_{x_0}/\Hr_{x_0}$ is infinite.
Let $R$ be the finite subset of $\Gr_{x_0}$ defined by
\[R=\{f_{\alpha(g v),x_0}g_vf_{\alpha(v),x_0}^{-1}, v\in\cup_{B\in\mathcal{F}}B, g\in S\}.\]
Since the group $\Gr_{x_0}$ is amenable, for
every $\epsilon>0$ there exists a subset
$F$ of $\Gr_{x_0}/\Hr_{x_0}$ such that $|\gamma F \Delta F|\leq \epsilon/N$ for all $\gamma \in R$. Since $\Gr_{x_0}/\Hr_{x_0}$ is infinite we may assume that $F$ does not contain the trivial
element $\Hr_{x_0}$.

Let $\widehat{\mathcal{F}}$ be the set of functions
$V\arr \coprod_{a \in A} \Gr_{a}/\Hr_{a}$ such that there exists
$B\in\mathcal{F}$ such that $\phi(v)=\Hr_{\alpha(v)}$ for
$v\notin B$, and $\phi(v)\in F$ for $v\in B$.
 Then $\widehat{\mathcal{F}}$ is split into a disjoint union of sets
$\widehat{\mathcal{F}}_B$ of functions with support equal to
$B\in\mathcal{F}$. We use here the fact that $F$ does not contain the trivial
element $\Hr_{x_0}$ of $\Gr_{x_0}/\Hr_{x_0}$. For each $B\in\mathcal{F}$ we have
\[\left|\widehat{\mathcal{F}}_B\right|=|F|^{|B|} = |F|^N,\]
hence
\[\left|\widehat{\mathcal{F}}\right|=
|\mathcal{F}|\cdot |F|^{N}.\]
The number of sets $B\in\mathcal{F}$ such that $g(B)\notin\mathcal{F}$
for some $g\in S$ is not larger than $\epsilon|\mathcal{F}|$. It
follows that the number of functions $\phi\in\widehat{\mathcal{F}}$
with support equal to such sets is not larger than
\[\epsilon|\mathcal{F}|\cdot |F|^{N-1}=\epsilon\left|\widehat{\mathcal{F}}\right|.\]

Let $g\in S$. Suppose that $B, g(B)\in\mathcal{F}$, and let
$\psi\in\widehat{\mathcal{F}}_B$. Take $v \notin B$. Then $\psi(v) =
\Hr_{\alpha(v)}$, and by our assumption that $\Sigma \subset B$, $g_v
\in \Hr_{\alpha(v)}$. By \eqref{eq:gactionphi} the support of
$g(\psi)$ is therefore a 
subset of $g(B)$. It follows that $g(\psi)$ does not belong to
$\widehat{\mathcal{F}}$ if and only if there exists $v\in B$ such that
$f_{\alpha(g v),x_0} g_v f_{\alpha(v),x_0}^{-1} \psi(v)\notin F$.
It follows that the cardinality of the set of elements
$\psi\in\widehat{\mathcal{F}}_B$ such that
$g(\psi)\notin\widehat{\mathcal{F}}$ is at most
\[\sum_{v \in B} |f_{\alpha(g v),x_0} g_v f_{\alpha(v),x_0}^{-1} F \setminus F| |F|^{N-1} < \epsilon |F|^N.\]

It follows that the cardinality of the set of functions
$\phi\in\widehat{\mathcal{F}}$ such that
$g(\phi)\notin\widehat{\mathcal{F}}$ for some $g\in S$ is not greater
than
\[\left(\epsilon+\epsilon\cdot|S|\right)\cdot
\left|\widehat{\mathcal{F}}\right|.\]
Since $\epsilon$ is an arbitrary positive number, it follows that the
action of $G$ on $\mathcal{P}$ is amenable.

The case when $\Gr_{x_0}/\Hr_{x_0}$ is finite can be reduced to the infinite
case, for instance, by the following trick. Replace $\Gr_{a}$ by
$\Gr_{a}\times\mathbb{Z}$, and define $\tilde{\mathcal{P}}$ as the set
of maps  $V\arr \coprod_{a \in A} (\Gr_{a}\times\mathbb{Z})/\Hr_{a}$ of
finite support, where support is defined in the same way as
before. Here $\Hr_{a}$ is considered to be the subgroup of
$\Hr_{a}\times\{0\}<\Gr_{a}\times\mathbb{Z}$.
Define the action of $G$ on $\tilde{\mathcal{P}}$ by the same
formulae~\eqref{eq:gv} and~\eqref{eq:gactionphi} as the action of $G$
on $\mathcal{P}$. Then, by the same arguments as above, there exists a
$G$-invariant mean on $\tilde{\mathcal{P}}$. The projection
$\coprod_{a \in A}\Gr_{a}\times\mathbb{Z}\arr \coprod{a \in A} \Gr_{a}$ induces a surjective
$G$-equivariant map $\tilde{\mathcal{P}}\arr\mathcal{P}$, which implies
that $\mathcal{P}$ has a $G$-invariant mean.
\end{proof}

In order to prove amenability of $G$, it remains to show that for
every $\phi\in\mathcal{P}$ the stabilizer $G_\phi$ of $\phi$ in $G$ is
amenable, see Proposition~\ref{pr:amenableactionandstabilizer}.
We will use a modification of an argument of Y.~de~Cornulier
from~\cite[Proof of Theorem~4.1.1]{cornulier:pleinstopologiques}. Suppose that it is not amenable.
It follows from Lemma~\ref{lem:recurrentamenablesubgroup} that for every
$v\in V$ the action of every subgroup of $G$ on the orbit of $v$ has
an invariant mean. Consequently, by Proposition~\ref{pr:amenableactionandstabilizer},
the stabilizer $G_{\phi, v}$ of $v$ in
$G_\phi$ is non-amenable. It follows by induction, that the intersection $K$
of the pointwise fixator of the support of $\phi$ with $G_\phi$ is
non-amenable. Since $K$ fixes all points $y\in\supp\phi$, we have a
homomorphism from $K$ to the direct product of the isotropy groups
$\Gr_y$ for $y\in\supp\phi$. If $g$ belongs to the kernel of this homomorphism,
then all germs of $g$ at points of $V$ belong to $\Hr$, hence
$g\in[[\Hr]]$, by Lemma~\ref{lem:singular}. It follows
that $K$ is an extension of a subgroup of $[[\Hr]]$ by a subgroup of a
finite direct product of isotropy groups $\Gr_y$. But this implies
that $K$ is amenable, which is a contradiction.
\end{proof}

\begin{remark}
Note that since conditions on the elements of $G$ in
Theorem~\ref{th:amenhomeo} are local, it
follows from the theorem that $[[\Gr]]$ is amenable if $\X$ is
compact.
\end{remark}

\section{Homeomorphisms of bounded type}

We start with description of groupoids $\Hr$ such that the full group
$[[\Hr]]$ is locally finite, which is in some sense the simplest class
of amenable groups. We will use then such groupoids as base for
application of Theorem~\ref{th:amenhomeo} and construction of
non-elementary amenable groups.

\subsection{Bratteli diagrams}

A \emph{Bratteli diagram} $\mathsf{D}=((V_i)_{i\ge 1}, (E_i)_{i\ge 1},
\be, \en)$ is defined by two sequences of finite sets $(V_i)_{i=1,
  2, \ldots}$ and $(E_i)_{i=1, 2, \ldots}$, and sequences of maps
$\be:E_i\arr V_i$, $\en:E_i\arr V_{i+1}$. We interpret the sets $V_i$
as sets of vertices of the diagram partitioned into
\emph{levels}. Then $E_i$ is the set of edges connecting vertices of
the neighboring levels $V_i$ and $V_{i+1}$. See~\cite{bratteli:af} for
their applications in theory of $C^*$-algebras. A \emph{path of length
  $n$}, where $n$ is a natural number or
$\infty$, in the diagram $\mathsf{D}$ is a sequence of edges
$e_i\in E_i$, $1\le i\le n$, such that $\en(e_i)=\be(e_{i+1})$ for
all $i$. Denote by $\paths_n(\mathsf{D})=\paths_n$ the set of paths of
length $n$. We will write $\paths$ instead of $\paths_\infty$.

The set $\paths$ is a closed subset of the
direct product $\prod_{i\ge 1}E_i$, and thus is a compact totally
disconnected metrizable space. If $w=(a_1, a_2, \ldots, a_n)\in\paths_n$ is a finite path of $\mathsf{D}$,
then we denote by $w\paths$ the set of all paths beginning
with $w$. Let $w_1=(a_1, a_2, \ldots, a_n)$ and $w_2=(b_1, b_2,
\ldots, b_n)$ be elements of $\paths_n$ such that
$\en(a_n)=\en(b_n)$. Then for every infinite
path $(a_1, a_2, \ldots, a_n, e_{n+1}, \ldots)$, the sequence $(b_1,
b_2, \ldots, b_n, e_{n+1}, e_{n+2}, \ldots)$ is also a path. The map
\begin{equation}
\label{eq:Tgamma}
T_{w_1, w_2}:(a_1, a_2, \ldots, a_n, e_{n+1}, \ldots)\mapsto (b_1,
b_2, \ldots, b_n, e_{n+1}, e_{n+2}, \ldots)
\end{equation}
is a homeomorphism $w_1\paths\arr w_2\paths$.

Denote by $\tail_{w_1, w_2}$ the set of germs of the homeomorphisms $T_{w_1, w_2}$. It is naturally identified with the set of all pairs $((e_i)_{i\ge
  1}, (f_i)_{i\ge 1})\in\paths^2$ such that $e_i=a_i$
and $f_i=b_i$ for all $1\le i\le n$, and $e_i=f_i$ for all $i>n$.

Let $\tail(\mathsf{D})$ (or just $\tail$) be the groupoid of germs of the semigroup generated by the transformations of the form $T_{w_1, w_2}$. It can be identified with the set of all pairs of cofinal paths, i.e., pairs of paths $(e_i)_{i\ge 1}$,
$(f_i)_{i\ge 1}$ such that $e_i=f_i$ for all $i$ big enough. The groupoid structure coincides with the groupoid structure of an equivalence relation:
the product $(w_1, w_2)\cdot (w_3, w_4)$ is defined if
and only if $w_2=w_3$, and then it is equal to $(w_1,
w_4)$. Here $\be(w_1, w_2)=w_2$ and $\en(w_1,
w_2)=w_1$. It follows from the definition of topology on a groupoid of germs that topology  on $\tail$ is given by the basis of open sets of the
form $\tail_{w_1, w_2}$. We call $\tail$ the \emph{tail equivalence groupoid} of the Bratteli diagram $\mathsf{D}$. More on the tail equivalence
groupoids, their properties, and relation to $C^*$-algebras, see~\cite{exelrenault:AF}.

Let us describe the topological full group
$[[\tail]]$. By compactness of $\paths$, for every $g\in[[\tail]]$
there exists $n$ such that for every $w_1\in\paths_n$ there exist a path
$w_2\in\paths_n$ such that
$g(w)=T_{w_1, w_2}(w)$ for all $w \in w_1 \Omega$. Then we say that \emph{depth} of $g$ is at
most $n$. For every $v\in V_{n+1}$ denote by $\paths_v$ the set of paths
$w\in\paths_n$ ending in $v$. Then the group of all elements
$g\in[[\tail]]$ of depth at most $n$ is naturally identified with the
direct product $\prod_{v\in V_{n+1}}\symm(\paths_v)$ of symmetric
groups on the sets $\paths_v$. Namely, if $g=(\pi_v)_{v\in
  V_{n+1}}\in\prod_{v\in V_{n+1}}\symm(\paths_v)$ for
$\pi_v\in\symm(\paths_v)$, then $g$ acts on $\Omega$ by the rule
$g(w_0w)=\pi_v(w_0)w$, where $w_0\in\paths_v$ and $w_0w\in\paths$.
Let us denote the group of elements of $[[\tail]]$ of depth at most
$n$ by $[[\tail]]_n$. We obviously have $[[\tail]]_n\le
[[\tail]]_{n+1}$ and $[[\tail]]=\bigcup_{n\ge 1}[[\tail]]_n$. In
particular, $[[\tail]]$ is locally finite (i.e., every its finitely
generated subgroup is finite).

The embedding $[[\tail]]_n\hookrightarrow [[\tail]]_{n+1}$ is
\emph{block diagonal} with respect to the direct product
decompositions $\prod_{v\in V_{n+1}}\symm(\paths_v)$ and $\prod_{v\in
  V_{n+2}}\symm(\paths_v)$. For more on such embeddings of direct
products of symmetric and alternating groups and their
inductive limits see~\cite{leinenpuglisi:1type,lavnek:diagonal}.

\subsection{Homeomorphisms of bounded type}
Let $\mathsf{D}$ be a Bratteli diagram. Recall that we denote by
$\paths_v$, where $v$ is a vertex of $\mathsf{D}$, the set of paths of
$\mathsf{D}$ ending in $v$ (and beginning in a vertex of the first
level of $\mathsf{D}$), and by $w\Omega$ we denote
the set of paths whose beginning is a given finite path $w$.

\begin{defi}
\label{def:alphav}
Let $F:\paths\arr\paths$ be a homeomorphism. For $v\in V_i$
denote by $\alpha_v(F)$ the number of paths $w\in\paths_v$ such
that $F|_{w\paths}$ is not equal to a transformation of the
form $T_{w, u}$ for some $u\in\paths_v$.

The homeomorphism $F$ is said to be \emph{of bounded type} if
$\alpha_v(F)$ is uniformly bounded and the set of points $w\in\paths$
such that the germ $(F, w)$ does not belong to $\tail$ is finite.
\end{defi}

It is easy to see that the set of all homeomorphisms of bounded type
form a group.

\begin{theorem}
\label{th:boundedamenable}
Let $\mathsf{D}$ be a Bratteli diagram. Let $G$ be a group acting
faithfully by homeomorphisms of bounded type on
$\paths(\mathsf{D})$. If the groupoid of germs of $G$
has amenable isotropy groups, then the group $G$ is amenable.
\end{theorem}

\begin{proof}
It is enough to prove the theorem for finitely generated groups $G$.
Since $[[\tail]]$ is locally finite, it is amenable. Therefore,
by Theorem~\ref{th:amenhomeo} applied with $\Hr = \tail \cap \Gr$,
where $\Gr$ is the groupoid of germs of $G$, it is enough to prove that the orbital
Schreier graphs of the action of $G$ on $\paths$ are recurrent.

Let $w\in\paths$, and let $S$ be a finite generating set of $G$.
Consider the Schreier graph of the orbit $G(w)$ of $w$. For $F\subset G(w)$ denote
by $\partial_S F$ the set of elements $x\in F$ such that $g(x)\notin
F$ for some $g\in S$.
\begin{lemma}
\label{lem:finbdry}
There exists an increasing sequence of finite subsets $F_n\subset G(w)$
such that $\partial_S F_n$ are disjoint and $|\partial_S F_n|$ is uniformly bounded.
\end{lemma}

\begin{proof}
The orbit $G(w)$ is the union of a finite number of $\tail$-orbits.
Let $P\subset\Gamma(S, w)$ be a $\tail$-orbit transversal. For
$u=(e_i)_{i\ge 1}\in P$ denote by $F_{n, u}$ the set of
paths of the form $(a_1, a_2, \ldots, a_n, e_{n+1}, e_{n+2},
\ldots)$. It is a finite subset of $G(w)$. Let $F_n=\bigcup_{u\in
  P}F_{n, u}$. Then $\bigcup_{n\ge 1}F_n=G(w)$.

Let $s\in S$. The number of paths $v=(a_1, a_2, \ldots, a_n,
e_{n+1}, e_{n+2}, \ldots)\in F_{n, u}$ such that $s(v)\notin
F_{n, u}$ is not greater than
$\alpha_{\be(e_{n+1})}(s)$, see Definition~\ref{def:alphav}.
It follows that $|\partial_S F_n|$ is not greater than
$|P|\cdot|S|\cdot\max\{\alpha_v(s)\;:\;s\in S, v\in\bigcup V_i\}$,
which is finite. We can assume that $\partial_S F_n$ are disjoint by
taking a subsequence.
\end{proof}

Applying Theorem~\ref{th:nashwilliams}, we conclude that the graph
$\Gamma(S, w)$ is recurrent. Note that this proof of recurrence of the graphs of actions
of homeomorphisms of bounded type is essentially the same as the proof
of~\cite[Theorem~V.24]{bondarenko:phd}.
\end{proof}

In the remaining subsections, we give some examples of applications of Theorem~\ref{th:boundedamenable}.

\subsection{Groups acting on rooted trees and bounded automata}

Let us start with the case when the Bratteli diagram
$\mathsf{D}$ is such that every level $V_i$
contains only one vertex. Then the diagram is determined by the
sequence $\alb=(E_1, E_2, \ldots)$ of finite sets of edges.
We will denote $\paths_n=E_1\times E_2\times\cdots\times E_n$ by $\alb^n$, and
$\paths=\prod_{i=1}^\infty E_i$ by $\xo$. The disjoint union $\xs=\bigsqcup_{n\ge 0}\alb^n$, where $\alb^0$ is a
singleton, has a natural structure of a rooted tree. The sets
$\alb^n$ are its levels, and two paths $v_1\in\alb^n$ and
$v_2\in\alb^{n+1}$ are connected by an edge if and only if $v_2$ is
a continuation of $v_1$, i.e., if $v_2=v_1e$ for some $e\in E_{n+1}$.

We present here a short overview of the notions and results on groups
acting on level-transitive rooted trees. For more details,
see~\cite[Section~6]{grineksu_en} and~\cite{nek:free}. Denote
by $\aut\xs$ the automorphism group of the rooted tree $\xs$.
The tree $\xs$ is level-transitive, i.e., $\aut\xs$ acts transitively
on each of the levels $\alb^n$. Denote by $\xs_{(n)}$ the tree of finite paths of the
``truncated'' diagram defined by the sequence $\alb_{(n)}=(E_{n+1}, E_{n+2},
\ldots)$.

For every $g\in\aut\xs$ and $v\in\alb^n$ there exists an
automorphism $g|_v\in\aut\xs_{(n)}$ such that
\[g(vw)=g(v)g|_v(w)\]
for all $w\in\xs_{(n)}$. The automorphism $g|_v$ is called the
\emph{section} of $g$ at $v$.

We have the following obvious properties of sections:
\begin{equation}
\label{eq:section}
(g_1g_2)|_v=g_1|_{g_2(v)}g_2|_v,\qquad g|_{v_1v_2}=g|_{v_1}|_{v_2}
\end{equation}
for all $g_1, g_2, g\in\aut\xs$, $v, v_1\in\xs$,
$v_2\in\xs_{(n)}$, where $n$ is the length of $v_1$.

The relation between groups acting on rooted trees and the
corresponding full topological groups of groupoids of germs is much
closer that in the general case of groups acting by homeomorphisms on
the Cantor set.

\begin{proposition}
\label{pr:fullgroupamenabilitytree}
Let $G$ be a group acting on a locally connected rooted tree $\xs$. Let
$\Gr$ be the groupoid of germs of the corresponding action on the
boundary $\xo$ of the tree. Then $[[\Gr]]$ is amenable if and
only if $G$ is amenable.
\end{proposition}

\begin{proof}
Since $G\le [[\Gr]]$, amenability of $[[\Gr]]$ implies amenability of
$G$. Let us prove the converse implication. Suppose that $G$ is
amenable. It is enough to prove that every finitely generated
subgroup of $[[\Gr]]$ is amenable. Let $S\subset[[\Gr]]$ be a finite set.
There exists a level $\alb^n$ of the tree $\xs$ such
that for every $v\in\alb^n$ and $g_i$ restriction of the action of
$g_i$ onto $v\xs_{(n)}$ is equal to restriction of an element of
$G$. Then every element $g\in S$ permutes the cylindrical sets
$v\xs_{(n)}$ for $v\in\alb^n$ and acts inside each of these sets as an
element of $G$. It follows that $\langle S\rangle$ contains a subgroup
of finite index which can be embedded into the direct product of a
finite number of quotients of subgroups of $G$. Consequently,
amenability of $G$ implies amenability of $\langle S\rangle$.
\end{proof}

\begin{defi}
An automorphism $g\in\aut\xs$ is said to be \emph{finitary} of
\emph{depth} at most $n$ if all sections $g|_v$ for $v\in\alb^n$ are trivial.
\end{defi}

It is an easy corollary of~\eqref{eq:section} that the set of finitary
automorphisms and the set of finitary automorphisms of depth at most
$n$ are groups. The latter group is finite, hence the group of all
finitary automorphisms of $\xs$ is locally finite. It is also
easy to see that the groupoid of germs of the action of the group of
finitary automorphisms on $\xo$ coincides with the tail equivalence
groupoid $\tail$ of the diagram $\mathsf{D}$.

\begin{defi}
\label{def:alphatree}
Let $g\in\aut\xs$. Denote by $\alpha_n(g)$ the number of paths
$v\in\alb^n$ such that $g|_v$ is non-trivial. We say that
$g\in\aut\xs$ is \emph{bounded} if the sequence $\alpha_n(g)$ is bounded.
\end{defi}

If $g\in\aut\xs$ is bounded, then there exists a finite set $P\subset\paths$
of infinite paths such that $g|_v$ is non-finitary only if $v$ is a
beginning of some element of $P$.

The following is proved in~\cite[Theorem~3.3]{nek:free}.

\begin{theorem}
\label{th:freesubgroups}
Let $G$ be a group of automorphisms of a locally finite rooted tree
$T$. If $G$ contains a non-abelian free subgroup, then either there
exists a free non-abelian subgroup $F\le G$ and a point $w$ of the
boundary $\partial T$ of the tree such that the stabilizer of $w$ in
$F$ is trivial, or there exists $w\in\partial T$ and a free
non-abelian subgroup of the group of germs $\Gr_w$.
\end{theorem}

In particular, if the orbits of the action of $G$ on $\partial T$ have
sub-exponential growth, and the groups of germs $\Gr_w$ do not
contain free subgroups (e.g., if they are finite), then $G$ does not
contain a free subgroup. We get therefore the following corollary of
Theorems~\ref{th:boundedamenable} and~\ref{th:freesubgroups}. (The
proof of the fact that there is no freely acting subgroup of $G$ under
conditions of Theorem~\ref{th:boundedtree} is the same as the proof
of~\cite[Theorem~4.4]{nek:free}, and also follow from Lemma~\ref{lem:finbdry}.)

\begin{theorem}
\label{th:boundedtree}
Let $G$ be a subgroup of the group of bounded automorphisms of $\xs$.
\begin{enumerate}
\item If the isotropy groups $\Gr_w$ are amenable, then the group $G$
is amenable.
\item If the isotropy groups $\Gr_w$ have no free subgroups, then the group $G$
has no free subgroups.
\end{enumerate}
\end{theorem}

In many cases it is easy to prove that the groups of germs $\Gr_w$ are
finite. Namely, the following proposition is straightforward
(see also the proof of~\cite[Theorem~4.4]{nek:free}).

\begin{proposition}
\label{pr:finiteisotropy}
Suppose that the sequence $|E_i|$ is bounded. Let $G$ be a group
generated by bounded automorphisms of $\xs$. If for every generator
$g$ of $G$ there exists a number $n$ such that the depth of every
finitary section $g|_v$ is less than $n$, then
the groups of germs $\Gr_w$ for $w\in\xo$ are locally finite.
Consequently, the group $G$ is amenable.
\end{proposition}

Amenability of groups satisfying the conditions of
Proposition~\ref{pr:finiteisotropy}
answers a question posed in~\cite{nek:free}. Below we show some concrete examples of groups satisfying the
conditions of Proposition~\ref{pr:finiteisotropy}.

\subsubsection{Finite automata of bounded activity growth}

\begin{defi}
\label{def:finitestate}
Suppose that the sequence $\alb=(E_1, E_2, \ldots)=(X, X, \ldots)$ is
constant, so that $\xs_{(n)}$ does not depend on $n$.
An automorphism $g\in\aut\xs$ is
\emph{finite-state} if the set
$\{g|_v\;:\;v\in\xs\}\subset\aut\xs$ is finite.
\end{defi}

The sequence $\alpha_n(g)$ from Definition~\ref{def:alphatree} for
finite-state automorphisms was studied by S.~Sidki
in~\cite{sid:cycl}. He showed that it is either bounded, or grows either as a polynomial of
some degree $d\in\mathbb{N}$, or grows exponentially (in fact, he
showed that the series $\sum_{n\ge 0}\alpha_n(g)x^n$ is rational). For
each $d$ the set $P_d(\xs)$ of automorphisms of $\xs$ for which $\alpha_n(g)$ grows as a
polynomial of degree at most $d$ is a subgroup of $\aut\xs$. He
showed later in~\cite{sidki:polynonfree} that these groups of
\emph{automata of polynomial activity growth} do not contain free
non-abelian subgroups. 

For different examples of subgroups of the groups of finite automata
of bounded activity, see~\cite[Section~1.D]{bkn:amenability} and
references therein.

It is easy to see that finite-state bounded automorphisms of the tree
$\xs$ satisfy the conditions of
Proposition~\ref{pr:finiteisotropy}. This proves, therefore, the
following.

\begin{theorem}
\label{th:boundedfinite}
The groups of finite automata of bounded activity growth are amenable.
\end{theorem}

This theorem is the main result of the
paper~\cite{bkn:amenability}. It was proved there by embedding
all finitely generated groups of finite-state bounded automorphisms
into a sequence of self-similar ``mother groups'' $M_d$, and then
using self-similarity structure on $M_d$ and studying entropy of the
random walk on $M_d$. Similar technique was applied in~\cite{amirangelvirag:linear} to prove
that the group of automata with at most linear activity growth is also
amenable. We will prove this fact later using Theorem~\ref{th:amenhomeo}.

One of examples of groups generated by finite automata of bounded
activity growth is the Grigorchuk group~\cite{grigorchuk:80_en}. It is
the first example of a group of intermediate growth, and also the
first example of a non-elementary amenable group. It was used later to
construct the first example of a finitely presented non-elementary
amenable group in~\cite{grigorchuk:notEG}. An uncountable family of generalizations of the Grigorchuk group were
studied in~\cite{grigorchuk:growth_en}. They all satisfy the
conditions of Proposition~\ref{pr:finiteisotropy}, and all have
sub-exponential growth.

Another example is the \emph{Basilica group}, which is the iterated monodromy group of the
complex polynomial $z^2-1$, and is generated by two automorphisms $a,
b$ of the binary rooted tree given by the rules
\begin{alignat*}{2}
a(0v) &= 1v, &\qquad a(1v) &= 0b(v),\\
b(0v) &= 0v, &\qquad b(1v) &= 1a(v).
\end{alignat*}
It is easy to see that $\alpha_n(a)$ and $\alpha_n(b)$ are equal to 1
for all $n$, therefore, by Proposition~\ref{pr:finiteisotropy}, the
Basilica group is amenable.

It was defined for the first time in~\cite{zukgrigorchuk:3st} as a
group defined by a three-state automaton. The authors showed that this
group has no free subgroups, and asked if it is amenable. They also
showed that this group can not be constructed from groups of
sub-exponential growth using elementary operations preserving
amenability (it is not \emph{sub-exponentially amenable}). Amenability
of the Basilica group was proved, using self-similarity
and random walks, in~\cite{barthvirag}. It follows from the results
of~\cite{nek:book} and Theorem~\ref{th:boundedfinite}
that iterated monodromy groups of sub-hyperbolic polynomials are
amenable.

\subsubsection{Groups of Neumann-Segal type}

The conditions of Definition~\ref{def:finitestate} are very
restrictive. In particular, the set of all finite-state automorphisms
of $X^*$ is countable, whereas the set of all bounded automorphisms is
uncountable. There are many interesting examples of groups generated by bounded but
not finite-state automorphisms of $X^*$. For example, the groups from the uncountable
family of Grigorchuk groups~\cite{grigorchuk:growth_en} are generated
by bounded automorphisms of the binary rooted trees. They are of sub-exponential growth.

An uncountable family of groups of non-uniformly exponential growth
was constructed in~\cite{brieussel:nonuniform}. All groups of the
family (if the degrees of the vertices of the tree are bounded)
satisfy the conditions of Proposition~\ref{pr:finiteisotropy}, hence
are amenable. Amenability of these groups were proved
in~\cite{brieussel:nonuniform} using the techniques
of~\cite{barthvirag}.

Other examples are given by a construction used by P.~Neumann
in~\cite{pneumann:brex} and D.~Segal in~\cite{segal:branch}.

Let us describe a general version of D.~Segal's construction.
Let $(G_i, X_i)$, $i=0, 1, \ldots$, be a sequence of groups acting transitively on finite
sets $X_i$. Let $a_{i, j}\in G_i$ for $1\le j\le k$ be sequences of
elements such that $a_{i, 1}, a_{i, 2}, \ldots, a_{i, k}$ generate
$G_i$. Choose also points $x_i, y_i\in X_i$. Define then automorphisms
$\alpha_{i, j}, \beta_{i, j}$, $i=0, 1, \ldots$,
of the tree $\xs_{(j)}$ for the sequence $\alb_{(j)}=(X_i,
X_{i+1}, \ldots)$ given by the following recurrent rules:
\[\alpha_{i, j}(xw)=a_{i, j}(x)w,\]
\[\beta_{i, j}(xw)=\left\{\begin{array}{ll}x_i\beta_{i+1, j}(w) &
    \text{if $x=x_i$,}\\ y_i\alpha_{i+1, j}(w) & \text{if $x=y_i$,}\\
xw & \text{otherwise,}\end{array}\right.\]
where $w\in\xs_{(j+1)}$ and $x\in X_i$.

\begin{proposition}
\label{pr:neumansegaltype}
The group $G=\langle\alpha_{0, 1}, \ldots, \alpha_{0, k}, \beta_{0, 1},
\ldots, \beta_{0, k}\rangle$ is amenable if and only if the group
generated by the sequences $(a_{1, j}, a_{2, j},
\ldots)\in\prod_{i=1}^\infty G_i$, for $j=1, \ldots, k$, is amenable.
\end{proposition}

\begin{proof}
The group generated by the sequences $(a_{1, j}, a_{2, j},
\ldots)$ is isomorphic to the group generated by $\beta_{0, 1},
\beta_{0, 2}, \ldots, \beta_{0, k}$. Consequently, if this group is
non-amenable, then $G$ is non-amenable too.

The automorphisms $\alpha_{i, j}$ are finitary. All sections of
$\beta_{i, j}$ are finitary except for the sections in finite
beginnings of the sequence $x_ix_{i+1}\ldots$. It follows that
$\beta_{i, j}$ are bounded, and the groups of germs the action of $G$
on $\xo$ are isomorphic to the group of germs of $G$ at the point
$x_0x_1\ldots$. It also follows from the description of the sections
that the group of germs of $G$ at $x_0x_1\ldots$ is a quotient of
$\langle\beta_{0, 1}, \beta_{0, 2}, \ldots, \beta_{0,
  k}\rangle$. Then Theorem~\ref{th:boundedtree} finishes the proof.
\end{proof}

The examples considered by P.~Neumann in~\cite{pneumann:brex} are
similar, but they are finite-state, so their amenability
follows from Theorem~\ref{th:boundedfinite}. The main examples of
D.~Segal~\cite{segal:branch} are non-amenable (they are constructed
for $G_i$ equal to $PSL(2, p_i)$ for an increasing sequence of primes $p_i$).

A.~Woryna in~\cite{woryna:cyclic} and
E.~Fink in~\cite{fink:nofree} consider the case when $G_i$ are cyclic
groups (of variable order). A.~Woryna uses the corresponding group $G$
to compute the minimal size of a topological
generating set of the profinite infinite wreath product of groups
$G_i$. E.~Fink shows that if the orders of $G_i$ grow sufficiently
fast, then the group $G$ is of exponential growth, but
does not contain free subgroups. Proposition~\ref{pr:neumansegaltype}
immediately implies that such groups are amenable (for any sequence of
cyclic groups).

I.~Bondarenko used in~\cite{bondarenko:fingeneration} the above
construction in the general case to study number of topological
generators of the profinite infinite wreath product of permutation groups.

\subsubsection{Iterated monodromy groups of polynomial iterations}
Two uncountable families of groups generated by bounded automorphisms were studied
in~\cite{nek:ssfamilies} in relation with holomorphic dynamics.
Let us describe one of them. For a sequence $w=x_1x_2\ldots$, $x_i\in\{0,
1\}$ define $s(w)=x_2x_3\ldots$, and let $\alpha_w, \beta_w,
\gamma_w$ be automorphisms of $\{0, 1\}^*$ defined by
\begin{alignat*}{3}
\alpha_w(0v)&=1v, &\qquad \gamma_w(0v)&=0v, &\qquad \beta_w(0v) &=0\alpha_{s(w)}(v), \\
 \alpha_w(1v)&=0\gamma_{s(w)}(v),&\qquad \gamma_w(1v)&=1\beta_{s(w)}(v),&\qquad \beta_w(1v)&=1v,
\end{alignat*}
if $x_1=0$, and
\[\beta_w(0v)=0v, \qquad \beta_w(1v)=1\alpha_{s(w)}(v),\]
if $x_1=1$, where $s(w)=x_2x_3\ldots$.

The automorphisms $\alpha_w, \beta_w, \gamma_w$ are obviously
bounded. Note that $\alpha_w, \beta_w, \gamma_w$ are finite-state if and only
if $w$ is eventually periodic. It is easy to see that the groups of germs of the action
of the group $R_w=\langle\alpha_w, \beta_w, \gamma_w\rangle$ on the
boundary of the binary tree are trivial. Therefore, it follows from
Theorem~\ref{th:boundedtree} that the groups $G_w$ are amenable.

It is shown in~\cite{nek:ssfamilies} that the sets of isomorphism classes
of groups $R_w$, for $w\in\{0, 1\}^{\mathbb{Z}}$ are countable, and
that the map $w\mapsto R_w$ is a homeomorphic embedding of the Cantor
space of infinite binary sequences into the space of three-generated
groups.

The groups $R_w$ are iterated monodromy groups of sequences of
polynomials of the form $f_n(z)=1-\frac{z^2}{p_n^2}$, where $p_n\in\mathbb{C}$
satisfy $p_n=1-\frac{1}{p_{n+1}^2}$, $n=0, 1, \ldots$. In general, let $f_1, f_2, \ldots$ be a sequence of complex polynomials seen as
maps
\[\mathbb{C}\stackrel{f_1}{\longleftarrow}\mathbb{C}\stackrel{f_2}{\longleftarrow}
\mathbb{C}\stackrel{f_3}{\longleftarrow}\cdots.\]

Denote by $P_n$ the union of the sets of critical values (i.e., values
at critical points) of the polynomials $f_n\circ
f_{n+1}\circ\cdots\circ f_{n+k}$ for all $k\ge 0$.
Note that $P_{n+1}\subseteq f^{-1}_{n}(P_n)$. We say that the sequence is \emph{post-critically finite} if
the set $P_1$ is finite.
In this case, we get a sequence of covering maps
\[M_1\stackrel{f_1}{\longleftarrow}M_2\stackrel{f_2}{\longleftarrow}M_3\stackrel{f_3}{\longleftarrow}
\cdots\]
where $M_1=\mathbb{C}\setminus P_1$ and $M_n=(f_1\circ\cdots\circ
f_{n-1})^{-1}(M_1)$. Choose a basepoint $t\in M_1$, and consider the
fundamental group $\pi_1(M_1, t)$.

The disjoint union
\[T=\{t\}\cup\bigsqcup_{n\ge 1}(f_1\circ\cdots\circ f_n)^{-1}(t)\]
has a natural structure of a rooted tree, where $t$ is the root, and a
vertex $z\in(f_1\circ\cdots\circ f_n)^{-1}(t)$ is connected to
$f_n(z)\in (f_1\circ\cdots\circ f_{n-1})^{-1}(t)$.

The fundamental group $\pi_1(M_1, t)$ acts on each level of the tree
by the monodromy action. Taken together, these actions become an
action of $\pi_1(M_1, t)$ on the tree $T$ by automorphisms. Denote by
$IMG(f_1, f_2, \ldots)$ the quotient of $\pi_1(M_1, t)$ by the kernel
of the action. This group is called the \emph{iterated monodromy
  group} of the sequence. For example, Basilica group is the iterated monodromy group of the
constant sequence $f_n(z)=z^2-1$. Here $P_1=\{0, -1\}$. Iterated monodromy groups of such sequences of polynomials were
studied in the article~\cite{nek:polynom}. The following statement is a direct
corollary of its results.

\begin{proposition}
Suppose that $|P_n|$ is a bounded sequence. Then there exists an
isomorphism of the tree $T$ with a tree $\xs$ conjugating
$IMG(f_1, f_2, \ldots)$ with a group acting on $\xs$ by bounded automorphisms.
\end{proposition}

It is also not hard to understand the structure of the groups of germs
of the action of $IMG(f_1, f_2, \ldots)$ on the boundary of the
tree. In particular, one can show for many sequences (using
Proposition~\ref{pr:finiteisotropy})
that they are finite.

\begin{examp}
Consider an arbitrary sequence $f_n(z)$, $n\ge 0$, such that
$f_n(z)=z^2$ or $f_n(z)=1-z^2$. Then the sequence is post-critically
finite with the post-critical sets $P_n$ subsets of $\{0, 1\}$. The
iterated monodromy group of such a sequence is generated by
automorphisms $a_w, b_w$, for some sequence $w=x_1x_2\ldots\in\{0,
1\}^\infty$, where
\[a_w(0v)=1v,\quad a_w(1v)=0b_{s(w)}(v),\quad b_w(0v)=0v,\quad
b_w(1v)1a_{s(w)}(v),\]
if $x_1=0$,
and
\[a_w(0v)=1v,\quad a_w(1v)=0a_{s(w)}(v),\quad b_w(0v)=0v,\quad
b_w(1v)1b_{s(w)}(v),\]
otherwise. All these groups are amenable by Proposition~\ref{pr:finiteisotropy}.
\end{examp}

\subsection{Bratteli-Vershik transformations and minimal homeomorphisms}

A \emph{Bratteli-Vershik} diagram is a Bratteli diagram
$\mathsf{D}=((V_n)_{n\ge 1}, (E_n)_{n\ge 1}, \be, \en)$ together with
a linear order on each of the sets $\en^{-1}(v)$, $v\in\bigcup_{n\ge
  2}V_n$. A path $v\in\paths_n$, for $n=1, 2, \ldots, \infty$, is said to be
\emph{minimal} (resp.\ \emph{maximal}) if it consists only of minimal
(resp.\ maximal) edges. Note that for every vertex $v\in\bigcup_{n\ge 1}V_n$ of a
Bratteli-Vershik diagram there exist unique maximal and minimal paths $v_{\max},
v_{\min}\in\paths_v$. Let $(e_1, e_2, \ldots)\in\paths$ be a non-maximal path in the
diagram. Let $n$ be the smallest index such that $e_n$ is
non-maximal. Let $e_n'$ be the next edge after $e_n$ in
$\en^{-1}(\en(e_n))$, and let $(e_1', e_2', \ldots, e_{n-1}')$ be the
unique minimal path in $\paths_{\be(e_n')}$. Define then
\[a(e_1, e_2, \ldots)=(e_1', e_2', \ldots, e_{n-1}', e_n', e_{n+1},
e_{n+2}, \ldots).\]
The map $a$ is called the \emph{adic transformation} of the
Bratteli-Vershik diagram.

The adic transformation is a continuous map from the set of non-maximal elements of
$\paths$ to $\paths$. In fact, it is easy to see that it is a
homeomorphism from the set of non-maximal paths to the set of
non-minimal paths. The inverse map is the adic transformation defined
by the opposite ordering of the diagram. If the diagram has a unique maximal and a unique minimal infinite
paths, then we can extend $a$ to a homeomorphism $\paths\arr\paths$ by
mapping the maximal path to the minimal path.

The following theorem is proved in~\cite{putn:bratel}.

\begin{theorem}
Every minimal homeomorphism of the Cantor set (i.e., a homeomorphism
for which every orbit is dense) is topologically conjugate to the adic
transformation defined by a Bratteli-Vershik diagram.
\end{theorem}

See also~\cite{medynets:aperiodic} where Vershik-Bratteli diagrams are
used to describe \emph{aperiodic} homeomorphisms, i.e., homeomorphisms
without finite orbits.

It follows directly from the definition of the adic transformation $a$
that $\alpha_v(a)=1$ for every vertex $v$ of the diagram, and that
germs of $a$ belong to $\tail$ for all points $w\in\Omega$ except for
the unique maximal path. It is also
obvious that the $a$ has no fixed points, hence the groups of germs of
the group generated by $a$ are trivial. Using
Theorem~\ref{th:boundedamenable}, we get a proof of
following result of~\cite{juschenkomonod} by K.~Juschenko and N.~Monod.

\begin{theorem}
\label{th:juschenkomonod}
Let $a$ be a minimal homeomorphism of the Cantor set. Then the
topological full group of the groupoid of germs of $\langle a\rangle$
is amenable.
\end{theorem}

It follows from the results of~\cite{medynets:aperiodic} that the same
result is true for all aperiodic homeomorphisms of the Cantor set. Theorem~\ref{th:juschenkomonod} provides the first examples of simple
finitely generated amenable groups. In fact, there are uncountably
many pairwise non-isomorphic full groups of minimal homeomorphisms of
the Cantor set, see~\cite{juschenkomonod}.

\subsubsection{One-dimensional tilings}
Adic transformations and subgroups of their full groups
naturally appear in the study of one-dimensional
aperiodic tilings, see~\cite{bellissardjuliensavinien}.
We present here one example, called \emph{Fibonacci tiling}.

Consider the endomorphism $\psi$ of the free monoid $\langle a, b\rangle$
defined by $a\mapsto ab,\qquad b\mapsto a$. Then $\psi^n(a)$ is beginning of $\psi^{n+1}(a)$, so that we can
pass to the limit, and get a right-infinite sequence
\[\psi^\infty(a)=abaababaabaababaababaabaababaabaab\ldots\]

Let $\mathcal{F}$ be the set of bi-infinite sequences $w\in\{a, b\}^{\mathbb{Z}}$ over the
alphabet $\{a, b\}$ such that every finite subword of $w$ is a
sub-word of $\psi^\infty(a)$. The set $\mathcal{F}$ is obviously
invariant under the shift, which we will denote by $\tau$. Here the
shift acts by the rule
\[\ldots x_{-2}x_{-1}.x_0x_1\ldots\mapsto \ldots x_{-1}x_0.x_1x_2,\]
where dot shows the place between the coordinate number -1 and
coordinate number 0.

Every letter $b$ in a sequence $w\in\mathcal{F}$ is uniquely grouped with the
previous letter $a$. Replace each group $ab$ by $a$, and each of the
remaining letters $a$ by $b$. Denote the new sequence by
$\sigma(w)=\ldots y_{-2}y_{-1}.y_0y_1\ldots$, so that $y_0$
corresponds to the group that contained $x_0$. It follows from the
definition of $\mathcal{F}$ that $\sigma(w)\in\mathcal{F}$.

\emph{Symbol} $\alpha(w)$ of a sequence $\ldots
x_{-2}x_{-1}.x_0x_1\ldots\in\mathcal{F}$ is an element of $\{a_0, a_1,
b\}$ defined by the following conditions:
\begin{enumerate}
\item $\alpha(w)=a_0$ if $x_{-1}x_0=aa$;
\item $\alpha(w)=a_1$ if $x_{-1}x_0=ba$;
\item $\alpha(w)=b$ if $x_0=b$.
\end{enumerate}

\emph{Itinerary} of $w\in\mathcal{F}$ is the sequence
$\alpha(w)\alpha(\sigma(w))\alpha(\sigma^2(w))\ldots$. We have the
following description of $\mathcal{F}$, which easily follows from the
definitions.

\begin{proposition}
\label{pr:fibonacci}
A sequence $w\in\mathcal{F}$ is uniquely determined by its
itinerary. A sequence $x_1x_2\ldots\in\{a_0, a_1, b\}^\infty$ is an
itinerary of an element of $\mathcal{F}$ if and only if
$x_ix_{i+1}\in\{a_0a_1, a_1a_0, a_1b, ba_0, ba_1\}$. The shift $\tau$
acts on $\mathcal{F}$ in terms of itineraries by the rules:
\begin{eqnarray*}
\tau(a_0w) &=& bw,\quad \tau(bw)= a_1\tau(w)\\
\tau(a_1w) &=& \left\{\begin{array}{ll} b\tau(w) & \text{if $w$
      starts with $a_0$,}\\
a_0\tau(w) & \text{if $w$ starts with $b$,}
\end{array}\right.
\end{eqnarray*}
\end{proposition}

Restrictions $\tau_{a_0}$, $\tau_{a_1}$, $\tau_b$ of $\tau$ to the
cylindrical sets  of sequences starting with $a_0$, $a_1$, $b$,
respectively, generate a \emph{self-similar inverse semigroup} in the
sense of~\cite{nek:smale}. The range of each of these
transformations does not intersect with the domain. Therefore, we may
define homeomorphisms $\alpha_0, \alpha_1, \beta$ equal to
transformations $\tau_{a_0}\cup\tau_{a_0}^{-1}$,
$\tau_{a_1}\cup\tau_{a_1}^{-1}$, $\tau_b\cup\tau_b^{-1}$ extended
identically to transformations of the space $\mathcal{F}$.

It is easy to see that the orbital Schreier graphs of the action of
the group $\langle\alpha_0, \alpha_1, \beta\rangle$ on $\mathcal{F}$
coincide with the corresponding two-sided sequences. Namely, the
Schreier graph of a point $\ldots x_{-1}.x_0x_1\ldots$ is isomorphic to
a chain of vertices $(\ldots, w_{-1}, w_0, w_1, \ldots)$, where
$w_0=w$, and the edge $(w_i, w_{i+1})$ corresponds to the generator
$\alpha_0, \alpha_1$, or $\beta$ if and only if $x_{i-1}x_i$ is equal to
$aa, ba$, or $ab$, respectively.
It follows from Theorem~\ref{th:juschenkomonod} that the group
$\langle\alpha_0, \alpha_1, \beta\rangle$ is amenable.

\subsection{An example related to the Penrose tilings}

Consider isosceles triangles formed by two diagonals and a side, and
two sides and a diagonal of a regular pentagon. Mark their vertices by
white and black dots, and orient one of the sides as it is shown on
Figure~\ref{fig:match}.
\begin{figure}
\centering
\includegraphics[scale=0.6]{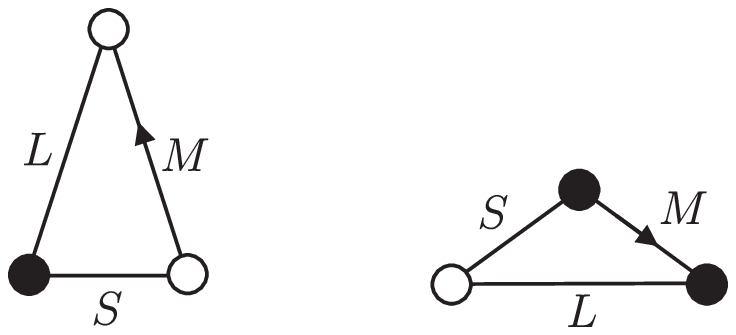}$\qquad\qquad\qquad\qquad$\includegraphics[scale=0.6]{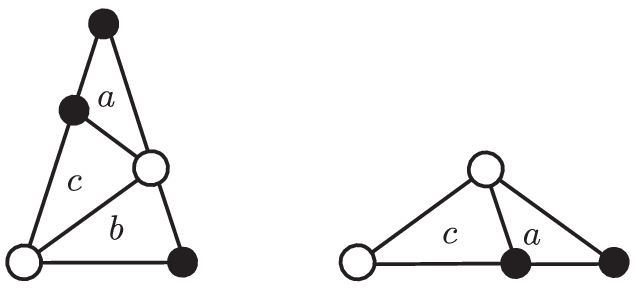}
\caption{Tiles of Penrose tilings and grouping of tiles}
\label{fig:match}
\end{figure}

A \emph{Penrose tiling} is a tiling of the plane by such triangles
obeying \emph{matching rules}: if two triangles have a common vertex,
then this vertex must be marked by dots of the same color; if two
triangle have a common side, then either they both are not oriented,
or they are both oriented and orientations match.

One can show, just considering all possible sufficiently big finite
Penrose tilings, that tiles of any Penrose tilings can be grouped into
blocks as on Figure~\ref{fig:match}. The blocks are triangles similar to the
original tiles (with similarity coefficient $(1+\sqrt{5})/2$), and the
new tiles also form a Penrose tiling, which we call \emph{inflation}
of the original tiling.

Given a Penrose tiling with a marked tile, consider the sequence of
iterated inflations of the tiling, where in each tiling the tile
containing the original marked tile is marked. Let $x_1x_2\ldots$ be
the corresponding \emph{itinerary}, where $x_i\in\{a, b, c\}$ is the
letter describing position of the marked tile of the $(i-1)$st
inflation in the tile of the $i$th inflation according to the rule
shown on Figure~\ref{fig:match}. A sequence is an itinerary of a
marked Penrose tiling if and only if it does not contain a subword
$ba$. More on the inflation and itineraries of the Penrose tilings,
see~\cite{tilings}.

Given a marked Penrose tiling with itinerary $w=x_1x_2\ldots$, denote
by $L(w)$, $S(w)$, $M(w)$ the itineraries of the same tiling in which
a neighboring tile is marked, where the choice of the neighbor is
shown on Figure~\ref{fig:match}. One can show by considering sufficiently big finite patches of the
Penrose tilings (see also~\cite{bgn,nek:ssinv}) that the transformations $L$, $S$,
and $M$ on the space of itineraries is given by the rules
\[S(aw)=cw,\quad S(bw)=bM(w),\quad S(cw)=aw,\quad L(cw)=cL(w),\]
\[L(aw)=\left\{\begin{array}{ll}bS(w) & \text{if $w$ starts with
      $a$,}\\ aM(w) & \text{otherwise,}\end{array}\right.
L(bw)=\left\{\begin{array}{ll}bS(w) & \text{if $w$ starts with $b$,}\\
aS(w) & \text{if $w$ starts with $c$,}\end{array}\right.\quad
\]
\[M(aw)=aL(w),\quad M(bw)=cw,\quad M(cw)=\left\{\begin{array}{ll}cM(w)
    & \text{if $w$ starts with $a$,}\\ bw &
    \text{otherwise.}\end{array}\right.\]

The orbital Schreier graphs of the action of the group $\langle L, S,
M\rangle$ coincide then with the graphs dual to the Penrose tilings
(except in the exceptional cases when the tilings have a non-trivial
symmetry group; then the dual graphs are finite coverings of the
Schreier graphs).

Let us redefine the transformations $L$, $S$, $M$,
trivializing the action on some cylindrical sets, so that in some
cases they correspond to moving the marking to a neighboring tile, but
sometimes they correspond to doing nothing. Namely, define new
transformations $L', S', M'$ by the rules.
\[S'(aw)=cw,\quad S'(bw)=bM'(w),\quad S'(cw)=aw,\]
\[L'(aw)=\left\{\begin{array}{ll}bS'(w) & \text{if $w$ starts with
      $a$,}\\ aw & \text{otherwise,}\end{array}\right.\]
\[L'(bw)=\left\{\begin{array}{ll}bS'(w) & \text{if $w$ starts with $b$,}\\
aS'(w) & \text{if $w$ starts with $c$,}\end{array}\right.\quad
L'(cw)=cL'(w),\]
\[M'(aw)=aw,\quad M'(bw)=cw,\quad M'(cw)=\left\{\begin{array}{ll}cM'(w)
    & \text{if $w$ starts with $a$,}\\ bw &
    \text{otherwise.}\end{array}\right.\]
Then the Schreier graphs of $\langle L', S', M'\rangle$ are subgraphs
of the graphs dual to the Penrose tilings. A piece of such Schreier
graph is shown on Figure~\ref{fig:penrose}.

\begin{figure}
\centering
\includegraphics{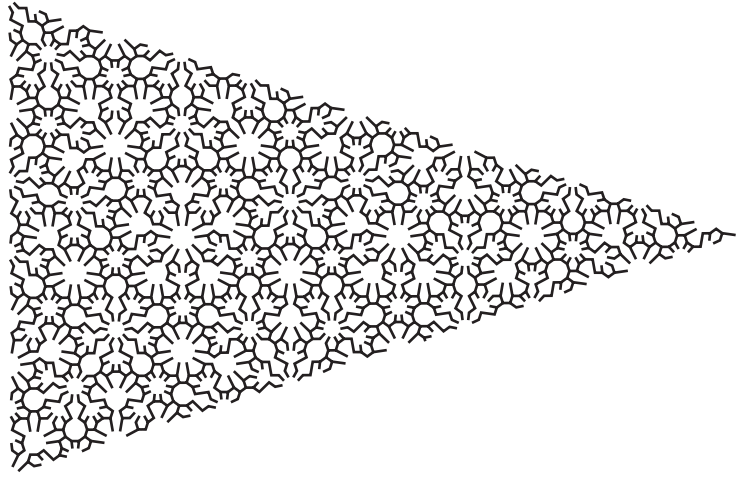}
\caption{A part of the Schreier graph of $\langle L', S', M'\rangle$}
\label{fig:penrose}
\end{figure}

It is easy to see that $M'$, and hence $S'$ are finitary. It follows
that $L'$ is bounded. Further analysis shows that the groups of germs
are finite, hence the group $\langle L', S', M'\rangle$ is amenable.

\section{Unbounded examples}
\label{s:unbounded}
In the previous section we applied Theorem~\ref{th:amenhomeo} in the case when the full
group $[[\Hr]]$ is locally finite. Here we present some other examples
of applications of Theorem~\ref{th:amenhomeo}.

\subsection{Linearly and quadratically growing automata}
Recall that $P_d(X)$, for $d=1, 2, \ldots$, denotes the group of
finite-state automorphisms $g$ of $\xs$ such that $\alpha_n(g)$ is
bounded by a polynomial of degree $d$. In particular, $P_1(\alb)$ is
the group of finite-state automorphisms of \emph{linear activity growth}. Here
$\alpha_n(g)$, as before, is the number of words $v$ of length $n$
such that $g|_v\ne 1$.

For example, consider the group generated by the following two
transformations of $\mathbb{Z}$.
\[a:n\mapsto n+1,\qquad b:2^m(2k+1)\mapsto 2^m(2k+3),\quad b(0)=0,\]
where $n, k, m\in\mathbb{Z}$ and $m\ge 0$.

The Schreier graph of this action, and the orbital Schreier graphs of
the associated group acting on a rooted tree were studied
in~\cite{benjaminihoffman,bcdn:schreier}.

The following theorem was proved in~\cite{amirangelvirag:linear}. We
present here a short proof based on Theorem~\ref{th:amenhomeo}.

\begin{theorem}
\label{th:linearamenable}
For every finite alphabet $X$ the group $P_1(X)$ is amenable.
\end{theorem}

\begin{proof}
Let us investigate the groups of germs of finitely generated subgroups
of $P_d(X)$. Let $S$ be a finite subset of $P_d(X)$, and $G=\langle S\rangle$. We may assume that $S$ is
\emph{state-closed}, i.e., that $g|_x\in S$ for all $x\in X$ and $g\in
S$. Replacing $X$ by $X^N$ for some integer $N$, 
we may assume that the elements of $S$ satisfy the following
condition. For every element $g \ne 1$ of the form $g=h|_x$, for $h\in S$ and
$x\in X$, there exists a unique letter $x_g\in X$
such that $g|_{x_g}=g$, for all other letters $x\in X$ the section
$g|_x$ has degree of activity growth lower than that of $g$ (in
particular, it is finitary if $g$ is bounded and trivial if $g$ is
finitary). This fact follows from the structure of automata of
polynomial activity growth~\cite{sid:cycl}, see also~\cite[proof of
Theorem~3.3]{bkn:amenability}.

In particular, it follows that for every $g\in P_d(X)$ the set of
points $w\in\xo$ such that the germ $(g, w)$ is not a germ of an
element of $P_{d-1}(X)$ is finite.

Let $w = x_1 x_2\ldots \in \xo$, and consider the isotropy group
$\Gr_w$ of the groupoid of germs of $G$. Denote by
$\Phi$ be the set of all eventually constant sequences. If $w$ does
not belong to $\Phi$, our assumptions on $S$ imply that the germ
$(g, w)$ of every $g\in S$ (and hence every $g \in G$)
belongs to the groupoid of germs of the finitary group. 
This implies that $\Gr_w$ is trivial. 
Assume now that $w$ belongs to $\Phi$, say $x_n=x$ for all $n$
large enough. The sequence $g|_{x_1x_2\ldots x_n}$ is eventually
constant for every $g \in S$ (and hence in $G$ by
\eqref{eq:section}). In particular if $g\in G$ fixes
$w$, then for all $y\neq x$ the sequence $g|_{x_1x_2\ldots x_ny}$ is eventually
constant and belongs to $P_{d-1}(X)$. We therefore get a homomorphism
from the stabilizer of $w$ in $G$ into the wreath product
$P_{d-1}(X)\wr\symm(X\setminus\{x\})$ with kernel the elements of $G$
that act trivially on a neighborhood of $w$. This shows that $\Gr_w$
embeds in $P_{d-1}(X)\wr\symm(X \setminus \{x\})$.

Now Theorem~\ref{th:amenhomeo} shows that the amenability of $G$ will
follow from amenability of $P_{d-1}(X)$
and recurrence of the orbital Schreier graphs of $G$. We have proved in
Theorem~\ref{th:boundedfinite} that $P_0(X)$ is amenable. Let us prove
amenability of finitely generated subgroups of $P_1(X)$, i.e.,
amenability of $P_1(X)$.

If $w_0 \notin \Phi$, then $w_0$ is not singular in the sense of
Theorem~\ref{th:amenhomeo}, i.e., germs of
elements of $P_1(X)$ at $w_0$ are equal to germs of some elements of
$P_0(X)$. 

We therefore only have to show that the action of $G$ on the orbit of
an element $w_0$ of $\Phi$ is recurrent. Let $F_n\subset\Phi$ be the
set of all sequences $w=x_1x_2\ldots\in\xo$ in the $G$-orbit of
$w_0$ such that
 $x_i=x_j$ for all $i, j>n$. Suppose that
$w\in\partial_SF_n$ for $n>1$. Then $w=vx^\omega$ for some $|v|=n$ and $x\in X$,
and there exists $g\in S$ such that $g(w)\notin F_n$, i.e.,
$g|_v(x^\omega)$ is not a constant sequence. This implies that
$g|_v\ne 1$, and $g|_{vx}\ne g|_v$, hence $g|_{vx}$ is bounded. Then either
$g|_v(x^\omega)$ is of the form $y_1y^\omega$, for $y_1, y\in X$; or $g|_{vxx}$ is
finitary, so that $g|_v(x^\omega)$ is of the form $y_1y_2y_3x^\omega$,
for $y_i\in X$. It follows that $|\partial_SF_n|\le|X|\sum_{g\in
  S}\alpha_n(g)$ and that every element of $g(\partial_SF_n)$ for
$g\in S$, belongs to
$F_{n+3}$. Then $\partial_SF_{3n}$ are disjoint, $\cup_n F_n = G(\omega_0)$ and $|\partial_SF_{3n}|$
is bounded from above by a linear function, which implies by
Theorem~\ref{th:nashwilliams} that the Schreier graph of $G$ on the orbit of $\omega_0$ is recurrent.
\end{proof}

Note that arguments of the proof of Theorem~\ref{th:linearamenable}
imply the following.

\begin{theorem}
\label{th:quadraticgrowth}
If $G\le P_2(X)$ is a finitely generated group with recurrent Schreier
graphs of the action on $\xo$, then $G$ is amenable.
\end{theorem}

It was announces in~\cite{amirvirag:higherdegree} that the Schreier
graphs of the ``mother group'' of automata of quadratic activity
growth are recurrent. This implies, by
Theorem~\ref{th:quadraticgrowth}, that these groups are amenable,
which in turn implies amenability of $P_2(X)$ for every finite
alphabet $X$.

On the other hand, it is shown in~\cite{amirvirag:higherdegree}
that the Schreier graphs of some groups
generated by automata of growth of activity of degree greater than 2
are transient.

\subsection{Holonomy groups of H\'enon maps}

A \emph{H\'enon map} is a polynomial map
$f:\mathbb{C}^2\arr\mathbb{C}^2$ given by
\[f(x, y)=(x^2+c-ay, x),\]
where $a, c\in\mathbb{C}$, $a\ne 0$. Note that $f$ is invertible, and its
inverse is $f^{-1}(x, y)=(y, (y^2+c-x)/a)$.

For a given H\'enon map $f$, let $K_+$ (resp.\ $K_-$) be the set of points
$p\in\mathbb{C}^2$ such that $f^n(p)$ is bounded as $n\to+\infty$
(resp.\ as $n\to-\infty$). Let $J_+$ and $J_-$ be the boundaries of
$K_+$ and $K_-$, respectively. The \emph{Julia set} of $f$ is the set
$J_+\cap J_-$.

In some cases the H\'enon map is \emph{hyperbolic} on a neighborhood
of its Julia set, i.e., the Julia set can be decomposed locally into a
direct product such that $f$ is expanding one and contracting the
other factors of the decomposition. The contracting direction will be
totally disconnected. See~\cite{ishii:hyperbolicI,ishii:hyperbolicII} for examples
and more detail. In this case we have a naturally defined holonomy pseudogroup of the
contracting direction of the local product decomposition of the Julia
set. Fixing a subset $C$ of the contracting direction homeomorphic to the
Cantor set, we can consider the group $H_C$ of all homeomorphisms of $C$
belonging to the holonomy pseudogroup, or a subgroup of $H_C$ such that
its groupoid of germs coincides with the holonomy groupoid. General
theory of holonomy pseudogroups of hyperbolic H\'enon maps (also their
generalizations of higher degree) is developed in~\cite{ishii:img}.

One example of computation of the holonomy groupoid is given in~\cite{oliva:phd}.
It corresponds to the H\'enon maps with parameter values $a=0.125$,
$c=-1.24$.

The corresponding group acts by homeomorphisms on the space $\{0,
1\}^\omega$ of binary infinite sequences, and is generated by
transformations $\alpha, \beta, \gamma, \tau$, defined by the
following recurrent rules:
\begin{alignat*}{2}
\alpha(0w) &=1\alpha^{-1}(w), &\qquad\alpha(1w)&=0\beta(w),\\
\beta(0w) &=1\gamma(w), &\qquad\beta(1w) &=0\tau(w),\\
\gamma(00w)&=11\tau^{-1}(w), &\qquad\gamma(11w)&=00\tau(w),\\
\gamma(10w)&=10w, &\qquad\gamma(01w)&=01w,\\
\tau(0w) &=1w, &\qquad\tau(1w) &= 0\tau(w).
\end{alignat*}
Note that $G=\langle\alpha, \beta, \gamma, \tau\rangle$ does not act by
homeomorphisms on the tree $\{0, 1\}^*$. The recurrent rules follow
from the automaton shown on~\cite[Figure~3.8]{oliva:phd}.

It is easy to check that the generators $\alpha, \beta$ have
linearly growing activity (i.e., the sequence from
Definition~\ref{def:alphav} for these homeomorphisms are bounded by a
linear function). The generators $\gamma, \tau$ are of bounded
type. The groups of germs of $G$ are trivial. Therefore,
Theorem~\ref{th:amenhomeo} applied for $\Hr$ equal to the groupoid of
germs of $\langle\gamma, \tau\rangle$ shows that $G$ is amenable.

\subsection{Mating of two quadratic polynomials}

Let $f(z)=z^2+c=z^2-0.2282\ldots+1.1151\ldots i$ be the quadratic polynomial
such that $f^3(0)$ is a fixed point of $f$. It is easy to see that the
last condition means that $c$ is a root of the polynomial
$x^3+2x^2+2x+2$. Let us compactify the complex plane $\mathbb{C}$ by a
circle at infinity (adding points of the form $+\infty\cdot
e^{i\theta}$) and extend the action of $f$ to the obtained disc $\mathbb{D}$
so that $f$ acts on circle at infinity by angle doubling. Take two
copies of $\mathbb{D}$ with $f$ acting on each of them, and then glue
them to each other along the circles at infinity using the map
$+\infty\cdot e^{i\theta}\mapsto +\infty\cdot e^{-i\theta}$. We get
then a branched self-covering $F$ of a sphere. This self-covering is a
\emph{Thurston map}, i.e., orbit of every critical point of $F$ (there
are two of them in this case) are finite. For more about this map in
particular, its relation to the paper-folding curve, and for
the general \emph{mating} procedure, see~\cite{milnor:dragons}.

The iterated monodromy group of $F$ was computed
in~\cite{nek:filling}. It is generated by the following five
automorphisms of the binary tree $\{0, 1\}^*$:
\begin{alignat*}{2}
a(0w) &= 1w, &\qquad a(1w) &= 0w, \\
b &= (1, a), &\qquad b' &= (cc'abb', 1),\\
c &=(c, b), &\qquad c' &=(c', b'),
\end{alignat*}
where $g=(g_0, g_1)$ means that $g(0w)=0g_0(w)$ and $g(1w)=1g_1(w)$.

It is shown in~\cite{nek:filling} that the subgroup generated by $a,
B=bb'$, and $C=cc'$ has a subgroup of index two isomorphic to
$\mathbb{Z}^2$. In particular, it is amenable, and Schreier graphs of
its action on $\{0, 1\}^\omega$ are recurrent. It is also shown
(see~\cite[Proposition~6.3]{nek:filling}) that the Schreier graphs of
the whole group $\mathrm{IMG}(F)=\langle a, b, c, b', c'\rangle$
coincide with the Schreier graphs of the virtually abelian
subgroup $H=\langle a, B, C\rangle$.
Namely, if $b(v)\ne v$, $b'(v)\ne v$, $c(v)\ne v$, or
$c'(v)\ne v$, then $b(v)=B(v)$, $b'(v)=B(v)$, $c(v)=C(v)$, or
$c'(v)=C(v)$, respectively. A picture showing how the Schreier graph
of $\langle a, b, c\rangle$ is embedded into the Schreier graph of $H$
can be found in~\cite[Figure 15]{nek:filling}.

It follows from the definition of the generators of $\mathrm{IMG}(F)$
that all germs of $b$ and $b'$ belong to the groupoid of germs $\Hr$ of
$H$. Consequently, all germs of $c$ and $c'$ except for the germ at
the point $000\ldots$ belong to $\Hr$. It follows from
Theorem~\ref{th:amenhomeo} and Proposition~\ref{pr:fullgroupamenabilitytree}
that the group $\mathrm{IMG}(F)$ is amenable.

\bibliographystyle{amsalpha}
\bibliography{nekrash,mymath}

\end{document}